\documentclass[a4paper]{article}
\usepackage{amsmath,amsthm,amssymb}
\usepackage[utf8]{inputenc}
\usepackage[english]{babel}

\usepackage{tikz}
\usepackage{graphicx}
\usetikzlibrary{matrix}
\usepackage{titlesec}
\usepackage[toc,page]{appendix}
\usepackage{listings}
\usepackage[nottoc,notlot,notlof]{tocbibind}

\newcommand{\tikzmark}[2]{\tikz[overlay,remember picture,baseline] \node [anchor=base] (#1) {$#2$};}

\newcommand{\DrawLine}[3][]{%
  \begin{tikzpicture}[overlay,remember picture]
    \draw[#1] (#2.north) -- (#3.south);
  \end{tikzpicture}
}

\newcommand{\DrawHLine}[3][]{%
  \begin{tikzpicture}[overlay,remember picture]
    \draw[#1] (#2.west) -- (#3.east);
  \end{tikzpicture}
}

\theoremstyle{plain} 
\newtheorem{theorem}{Theorem} 
 
\newtheorem{lemma}[theorem]{Lemma}

\newtheorem{example}[theorem]{Example}

\theoremstyle{definition} 
\newtheorem{definition}[theorem]{Definition}

\usepackage{blkarray}
\usepackage[square,numbers]{natbib}
\usepackage{parskip}
\title{Existence of Projective Planes}
\date{\today}
\author{Xander Perrott}

\begin{document}

\maketitle

\begin{abstract}
This report gives an overview of the history of finite projective planes and their properties before going on to outline the proof that no projective plane of order 10 exists. The report also investigates the search carried out by MacWilliams, Sloane and Thompson in 1970 \cite{B15C} and confirms their result by providing independent verification that there is no vector of weight 15 in the code generated by the projective plane of order 10.
\end{abstract}

\section{Introduction}

This report introduces the concept of projective planes and gives an understanding of their properties. Section 2 details the history of research into the existence of projective planes including an overview of the search for the projective plane of order 10. This is developed further in Section 3 where we outline the method and theory used in the search for the projective plane of order 10. Finally, Section 4 details how $A_{15}$ was shown to be 0  and gives the results of an independent search for code words of weight 15.

\subsection{Properties of Projective Planes}
Projective planes were defined in order to develop the notion of a plane to include the additional property that any two distinct lines should have a unique intersection point. At this point we are thinking of a plane as a flat surface of two dimensions and a line as being a one dimensional subspace that is embedded in this plane, but we should redefine lines as sets of points contained in the plane. In a plane which is not projective two lines can be parallel and each contain a distinct set of points. They do not have a point in common at which they intersect. In a projective plane we will ensure that for every pair of distinct lines there is exactly one point that belongs to both.\\

\begin{example} The Euclidean plane is not a projective plane as we can find  pairs of parallel lines, e.g. $x=0$ and $x=1$. We can create a projective plane by starting with the Euclidean plane and adding points to make sure every distinct pair of lines has a unique intersection. Currently a set of lines with the same gradient are all parallel to each other and do not intersect so we take all the lines of a particular gradient, called a parallel class, and add the same point to each of them. Now all of the lines of the parallel class intersect at this new point. This new point does not lie on the in the plane and is called a `point at infinity'. We add a point at infinity for every parallel class so there are now an infinite number, one for every value the gradient could take and then add one additional line which contains every point at infinity and is known as the `line at infinity'. Now the plane is a projective plane as all pairs of distinct lines have a single intersection point, there are only three cases to check. \\
1. Two distinct lines which are not parallel (and neither is the line at infinity) have one intersection in the Euclidean plane as before.\\
2. Two distinct lines which are parallel meet on the line at infinity as they are in the same parallel class.\\
3. The line at infinity and any other line meet in a point on the line at infinity as every line belongs to some parallel class and therefore has one point on the line at infinity.
Note that by adding the line at infinity we have retained the important property of planes that any two points of the plane are both contained in a line of the plane.\\
\end{example}

It is natural for us to choose to extend the Euclidean plane in this way since it is the plane we are most familiar with. However, it is not particularly manageable since it has an infinite number of points and lines. In order to gain further understanding about projective planes we need to find planes with finite numbers of points and lines. By limiting the projective plane to a finite, but nontrivial size we get an interesting result; the number of lines that pass through each point is equal and furthermore is equal to the number of points that are contained in each line. This can be seen by considering a point $p$, contained in $k$ lines in a projective plane. From our requirement that the plane is nontrivial we assume there exists a quadrangle, a set of four points where no three are contained in one line. Thus we can take $l$, a line that doesn't contain $p$, and each of the $k$ lines through $p$ must intersect $l$ in exactly one point. Each of the $k$ points of intersection must be distinct as the lines already have a unique intersection at $p$ so $l$ contains at least $k$ points. If $l$ contains more than $k$ points then there is some point, $x$, on $l$ that is not on one of the $k$ lines. This is not possible as it suggests there is no line containing $p$ and $x$, a contradiction to the properties of planes. Now consider a point, $q$, also not on the line $l$, and then there must be $k$ lines containing $q$ paired with each of the $k$ lines of $l$. Any more than $k$ lines containing $q$ would suggest that the additional lines do not intersect $l$; again, a contradiction. Repeating this argument for all points and lines gives the result.

This gives us an easy way of classifying finite projective planes based on their size.\\

\begin{definition} A projective plane of order $n$ where $n\geq2$ is a finite set of points and lines (defined as sets of points), such that:
\begin{enumerate}
\item Every line contains $n+1$ points
\item Every point lies on $n+1$ lines
\item Any two distinct lines intersect in a unique point
\item Any two distinct points lie on a unique line.\\
\end{enumerate}
\end{definition}
This definition is the same under duality meaning that if we change statements into their dual statements then the definition is unchanged. A dual statement is achieved by switching words relating to lines such as `line', `intersect' and `contains' with an appropriate choice of word relating to points, `point' and `lies'. This changes  which statements are about lines and which are about points.\\

\begin{center}
\includegraphics[scale=0.5]{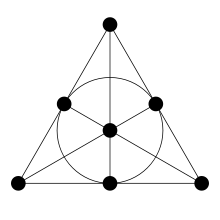}
\end{center}

\begin{example} The projective plane of order 2 (The Fano plane) is shown here. There are seven lines (six straight and one curved) and seven points with each line containing three $(2+1)$ points and each point lying on three lines. This also gives the plane a high degree of symmetry as by relabeling the points it can be seen that all the lines are equivalent in the structure. Additionally we can choose to view the points as lines of the plane and the lines as points and the diagram still represents the same structure (by the duality mentioned previously). The plane is said to be non-Euclidean as there is no way to draw the configuration in the Euclidean plane without the use of a curved line.\\
\end{example}

\begin{lemma}In a projective plane of order $n$ there are $n^2+n+1$ points and $n^2+n+1$ lines.
\end{lemma}

\begin{proof} Consider a projective plane of order $n$ and take any point of the plane, $p$. There are $n+1$ lines passing through $p$ and each of these lines has $n$ points other than $p$ on it. This accounts for all points in the plane since for any two distinct points there must be a line containing both of them, therefore any point which is not $p$ must lie on a line that passes through $p$. Note also, if some point $x$ lies on more than one line passing through $p$ (and would therefore be counted twice in this method of counting the points) this would cause two lines through $p$ to intersect in two points, $p$ and $x$, which contradicts our definition of a projective plane. This means we have $n+1$ lines times $n$ points and the point $p$, resulting in $n(n+1)+1=n^2+n+1$ points in the projective plane of order $n$. By duality a similar argument can be made to show there are also $n^2+n+1$ lines in the projective plane of order $n$.\end{proof}

\subsection{Incidence Matrices}
Another way for us to view projective planes is as an incidence matrix. The rows of the matrix correspond to lines of the plane and the columns correspond to points in the plane so we need a matrix of size $(n^2+n+1)\times (n^2+n+1)$. The matrix tells us which points are contained in each line and which lines lie on each of the points. We define the entries of the matrix as 
$$ A_{ij}=\begin{cases} 
      1 & \textnormal{line }i\textnormal{ contains point }j \\
      0 & \textnormal{otherwise}
   \end{cases}$$
The incidence matrix needs to represent the structure of a projective plane of order $n$ so it should have have the properties:
\begin{enumerate}
\item Every row contains $n+1$ ones
\item Every column contains $n+1$ ones
\item Two distinct rows have exactly one column where they both have one as an entry
\item Two distinct columns have exactly one row where they both have one as an entry.
\end{enumerate}

The existence of such a matrix of size $(n^2+n+1)\times (n^2+n+1)$ is equivalent to the existence of a projective plane of order $n$. \\

\begin{example} Shown below are two incidence matrices for the Fano plane. There are multiple representations of any incidence matrix as there are many possible ways to label the points and lines of the plane which changes the order of rows and columns in the matrix. Naturally, the two matrices here are the same up to switching rows or columns as there is only one representation of the Fano plane excluding relabeling. Matrix A is an intuitive way of constructing the incidence matrix by going through the points sequentially and listing all lines on the point that have not yet been expressed in the matrix. Matrix B has a nicer cyclic representation, but it is not possible to find such a representation for every finite projective plane. 
$$A=\begin{pmatrix}
1 & 1 & 1 & 0 & 0 & 0 & 0\\
1 & 0 & 0 & 1 & 1 & 0 & 0\\
1 & 0 & 0 & 0 & 0 & 1 & 1\\
0 & 1 & 0 & 1 & 0 & 1 & 0\\
0 & 1 & 0 & 0 & 1 & 0 & 1\\
0 & 0 & 1 & 1 & 0 & 0 & 1\\
0 & 0 & 1 & 0 & 1 & 1 & 0
\end{pmatrix}\quad B=\begin{pmatrix}
1 & 1 & 0 & 1 & 0 & 0 & 0\\
0 & 1 & 1 & 0 & 1 & 0 & 0\\
0 & 0 & 1 & 1 & 0 & 1 & 0\\
0 & 0 & 0 & 1 & 1 & 0 & 1\\
1 & 0 & 0 & 0 & 1 & 1 & 0\\
0 & 1 & 0 & 0 & 0 & 1 & 1\\
1 & 0 & 1 & 0 & 0 & 0 & 1
\end{pmatrix}$$
\end{example}
These are incidence matrices for the Fano plane (order 2) so the number of points and lines is $n^2+n+1=4+2+1=7$ hence these are $7\times 7$ matrices. There are three ones in each row and column representing the fact that every line contains three points and every point lies on three lines of the plane. Let $u$ and $v$ be rows of the incidence matrix where $u\neq v$. Then $u\cdot v=1$ as there is only one point in common on lines $u$ and $v$. We also have $v\cdot v=3$ as all lines have three points on them so taking the inner product of two of the same rows gives three points where both lines have one as an entry. Thus we can deduce
$$AA^{T}=2I+J$$
where $J$ is the $7\times 7$ matrix with 1 for every entry, $I$ is the identity matrix and $A^T$ is the transpose of $A$. This is a direct result of the inner products we have just calculated. On the main diagonal the values are the inner product of two identical rows giving the value of three. Off the diagonal each entry is one since  we are taking the inner product of distinct rows. In general, where $A$ is the incidence matrix of the projective plane of order $n$
$$AA^{T}=nI+J.$$

Something that will be of interest to us in a later section is the determinant of $A$. Let $C=nI+J$ so that $AA^T=C$. Then we have
$$\det(C)=\det(AA^T)=\det(A)\cdot \det(A^T)=(\det(A))^2.$$
Since we have a way of relating them it is easiest to find $\det(A)$ by finding $\det(C)$ first.
$C$ is a $(n^2+n+1)\times (n^2+n+1)$ matrix with entries of $n+1$ on the diagonal and one everywhere else, therefore we can calculate the following:
$$\begin{vmatrix}
n+1& 1 & \dots & 1 \\
1 &n+1 & & \vdots \\
\vdots & & \ddots & 1 \\
1 & \dots & 1 & n+1
\end{vmatrix}
= \pm \begin{vmatrix}
n+1 & 1 & \dots & 1 \\
-n & n & & 0 \\
\vdots & & \ddots & \vdots \\
-n & 0 & \dots & n
\end{vmatrix}
= \pm \begin{vmatrix}
(n+1)^2& 1 & \dots & 1 \\
0 & n & & 0 \\
\vdots & & \ddots & \vdots \\
0 & \dots & 0 & n
\end{vmatrix}$$
The first equality comes from subtracting the first row from every other, an operation which preserves the determinant. The second equality is the result of adding every other column to the first one. Again this operation does not change the determinant. Note that the top left entry is given by $(n+1)+(n^2+n)\cdot 1=n^2+2n+1=(n+1)^2$.
Finally, the determinant can be equated by taking the product of the diagonal elements of the last matrix since it is in upper triangular form. Therefore
$$\det(C)=(n+1)^2n^{n^2+n}$$
$$\det(A)=\sqrt{\det(C)}=\pm(n+1)n^{(n^2+n)/2}.$$

\subsection{Vector Spaces}
\begin{definition}Let A be a set of vectors, $\{a_1,a_2,\ldots, a_n\}$. The vector space generated by $A$ over a field $F$ is given by
$$\{ x \mid x = f_1a_1+f_2a_2 + \cdots + f_na_n , f_i \in F \}.$$  \end{definition}
We will be working in binary so $F=\mathbb{Z}/2\mathbb{Z}=\{0,1\}$ where $1+1=0$. This means that the generated vectors are actually the sum of a subset of vectors from $A$. 
The reason we introduce vector spaces here is that we want to consider the vector space generated by the rows of the incidence matrix of a projective plane. Since the rows of the incidence matrix correspond to lines of the projective plane we are actually looking for all configurations of points which are the sum of lines of the plane. We can also call this vector subspace a code and vectors in the code are called code words. We say a point is on a configuration if the entry of the code word corresponding to the point is a one. \\

\begin{example} We can find the binary code generated by the rows of the incidence matrix for the projective plane of order 2 using matrix B above. Any vector generated is a sum of rows in B so first note that if we take the sum of the empty set of rows we get the vector (0,0,0,0,0,0,0) in the code. Next consider three lines which all pass through the same point, $p$. Adding these three rows gives (1,1,1,1,1,1,1) since every point is on this configuration of lines; $p$ is on three lines and every other point lies on a single line.
One line configurations contain three points as they are naturally the rows making up B, this gives us $$\{(1,1,0,1,0,0,0), (0,1,1,0,1,0,0), \ldots , (1,0,1,0,0,0,1)\}.$$  Two line configurations contain four points since each line contains three points, but they have one point of intersection, $p$. $p$ therefore is not on the configuration, but the other two points on each line are giving four points in total. Additionally, there is always a third line passing through $p$ and the sum of all three lines is (1,1,1,1,1,1,1) so the sum of two lines intersecting at $p$ is always the complement of the points on the third line through $p$. This gives us seven more vectors in the vector space $$\{(0,0,1,0,1,1,1), (1,0,0,1,0,1,1), \ldots ,(0,1,0,1,1,1,0)\}.$$ Any third line added to the sum of two lines will either result in (1,1,1,1,1,1,1) or another line of the plane. Thus any further addition of lines does not generate anything beyond the 16 vectors we have found so they are only vectors generated by the incidence matrix.\\
\end{example}

When we take the sum of a number of code words, a point $a$ is on the resulting code word if it was on an odd number of the summed configurations. Since each entry of the code words is summed independently if $a$ is in $k$ of the summed code words then the entry of $a$ is given by $k\cdot 1=k$. If $k$ is even then $k=0$ (mod 2) and $a$ is not on the configuration whereas if $k$ is odd then $k=1$ (mod 2) so $a$ is on the configuration. \\

\begin{definition}
The weight of a code word $v$ is the number of non-zero components of the code word. It is denoted $w(v)$.
\end{definition}
This is equivalent to the number of ones in the code word since we are only using a binary code. Therefore the weight of the code word is also the number of points in the configuration.\\

\begin{example} If $v=(1,1,0,1,0,0,0)$, $w(v)=3$. \\
\end{example}

\begin{lemma}Let $C$ be a binary code and $v,u \in C$. Then $$w(v+u)=w(v)+w(u)-2w(v\cap u).$$ \end{lemma}

\begin{proof} Points that are in the intersection of $v$ and $u$ are on an even number of summed configurations so are not in $v+u$. Points that are on either $v$ or $u$ however are on the configuration $v+u$. Thus we have
\begin{align*} w(v+u)&=w(v\backslash u)+w(u\backslash v) \\
&= (w(v\backslash u)+w(v\cap u))+(w(u\backslash v)+w(u\cap v))-2w(v\cap u) \\
&= w(v)+w(u)-2w(v\cap u).
\end{align*}
\end{proof}

\begin{definition}The weight-enumerator polynomial of a code $C$ is given by
$$W_C(x,y)=\sum_{i=0}^{N}A_ix^{N-i}y^i$$
where $N=n^2+n+1$ and $A_i$ is the number of vectors of weight $i$ in $C$.\\
\end{definition}

\begin{example} Continuing our example of the vector space $V$ generated by the incidence matrix of the Fano plane we want to find its weight-enumerator polynomial. Previously we saw that it has one weight zero vector, seven weight three vectors corresponding to lines of the plane, seven weight four vectors representing complements of lines and one weight seven vector. Using the previous definition this gives us
$$W_V(x,y)=x^7+7x^4y^3+7x^3y^4+y^7.$$
\end{example}

%%%%%%%%%%%%%%%%%%%%%%%%%%%%%%%% SECTION 2 %%%%%%%%%%%%%%%%%%%%%%%%%%%%%%%%%%%%%%%%%

\newpage
\section{History of Projective Planes}
The concept of finite projective planes has been around since the very early 1900s \cite{A32}. It was of interest to mathematicians to determine the orders for which planes could exist. The strongest proof of existence was that projective planes of order $n$ exist whenever $n$ is a prime power i.e. for any $n=p^q$, where $p$ is prime and $q$ is an integer \cite{D}.\\

This left the projective plane of order 6 as the smallest whose existence was unknown. In the 1930s, Bose came up with a condition that if a projective plane of order $n$ exists then $n-1$ orthogonal Latin squares of order $n$ must exist \cite{A4}. Euler \cite{Euler} had been interested in orthogonal Latin squares in 1782 and created the 36 officer problem requiring two orthogonal Latin squares of size 6 which he believed was unsolvable. He went on to theorise that for any $n=2$ (mod 4) orthogonal Latin squares of order $n$ do not exist, but this was only certain for $n=2$ at the time. Euler's conjecture has since been proven false, but he was correct in thinking his 36 officer problem had no solution. In 1901, Tarry hand checked all Latin squares of order 6 to find that no two were mutually orthogonal, let alone five \cite{B18}. In conjunction with Bose's result this proved the non-existance of the projective plane of order 6.\\

The Bruck-Ryser Theorem \cite{A7} confirmed this result, an improvement over enumeration which is fallible. The theorem said that for $n=1,2$ (mod 4) a projective plane of order $n$ can only exist when $n$ is the sum of two squares. As six is not the sum of two squares the projective plane of order 6 must not exist.
This result however did not say anything about the existence of the projective plane of order 10 as $10=3^2+1^2$.\\

Searching for a projective plane of order 10 is equivalent to finding a suitable $111\times 111$ incidence matrix, A, where $AA^{T}=10I+J$ ensuring the properties of projective planes hold.
The reasoning and computer search that showed no projective plane of order 10 existed began to take shape in the early 1970s. Lam \cite{A} explains that a talk given by Assmus introduced the concept of working with vector spaces which increased interest in the problem. Maybe the code space would return an easily reached contradiction proving non-existence.\\

The weight-enumerator polynomial for the vector space generated by the projective plane of order 10 was a natural construction as it gave information about the structure of the plane. MacWilliams, Sloane and Thompson \cite{B15C} explain that Assmus' article \cite{C1E} and 1970 talk, along with unpublished reasoning by Thompson, show that the weight-enumerator can be expressed in terms of $A_i$ for $i=12$, $15$ and $16$. This means that once the numbers of vectors of weight 12, 15 and 16 in the vector space are found the entire weight-enumerator polynomial is known.\\

The first of these to be calculated was $A_{15}$, shown to be zero by MacWilliams, Sloane and Thompson in an exhaustive computer search \cite{B15C}. This was shown again by Bruen and Fisher \cite{B5} using a different method based on previous work by Denniston \cite{B7}. In 1983, Lam, Thiel and Swiercz found that $A_{12}=0$ \cite{B10} and proceeded to finish the search for $A_{16}$ started by Carter to demonstrate that $A_{16}$ was also zero \cite{B12}.\\

Once $A_{12}$, $A_{15}$ and $A_{16}$ were all known to be zero the weight-enumerator polynomial could be calculated. It was seen that $A_0=0$ and $A_{11}=111$ as expected and the next smallest non-zero vector weight is $A_{19}=24,675$. This means that the incidence matrix for the projective plane of order 10 should be able to be constructed starting with a 19 point configuration. When Lam, Thiel and Swiercz ruled out all such 19 point configurations as starting vectors in the space \cite{B} it proved that no incidence matrix could exist. Therefore the projective plane of order 10 does not exist.

%%%%%%%%%%%%%%%%%%%%%%%%%%%%%%% SECTION 3 %%%%%%%%%%%%%%%%%%%%%%%%%%%%%
\newpage
\section{Specifics of the non-existence argument}

In this section we will go through a more rigorous explanation of the method that was used to show the projective plane of order 10 does not exist. We will follow the structure of K{\aa}rhstr{\"o}m \cite{D}.

\subsection{Theories of Projective Plane Codes}
Whereas the properties of projective planes discussed in the introduction applied to all orders, here we will present some more specific theories relating to the projective plane of order 10.\\

\begin{lemma}Let $x$ be a code word of the code generated by a finite projective plane of order 10. A line, $l$, of the projective plane intersects $x$ in an odd number of points if and only if $x$ is a sum of an odd number of lines of the projective plane.
\end{lemma}

\begin{proof} First we consider the case where $l$ is not one of the lines in $x$. Let $x$ be the sum of $i$ lines where $i$ is an odd number. As the line, $l$, intersects each of the lines that sum to $x$ precisely once, there are $i$ intersections of $l$ with lines of $x$. Let $j$ be the number of points on $l$ where an odd number of lines of $x$ intersect and $11-j$ be the number of points where an even number intersect. Then $i=1\cdot j+0\cdot (11-j)$ (mod 2) so $j$ must be odd.
If an odd number of the lines of $x$ intersect $l$ at a point $a$ then $a$ is in $x$ since it is on an odd number of lines of $x$.
Similarly the points where an even number of lines of $x$ intersect $l$ are not in $x$. Therefore the number of points where $l$ intersects $x$ is $j$ which is odd.

Retaining the same notation, if the number of places $l$ intersects $x$ is odd then so is $j$, the number of points where an odd number of lines of $x$ intersect $l$. The equation $i=1\cdot j+0\cdot (11-j)$ tells us that $i$ is also odd so $x$ is the sum of an odd number of lines.

Now let $x$ be the sum of an odd number of lines including $l$. Then consider the configuration $x-l$ which is the sum of an even number of lines. Since $l$ is not a line of $x-l$, $l$ intersects $x-l$ in an even number of points (from previous case). Points of $l$ that are on $x-l$ lie on an odd number of lines in $x-l$, but therefore lie on an even number of lines in $x$ so are not in $x$. Similarly, points of $l$ that are not on $x-l$ are on $x$. Therefore the points of $l$ that intersect with $x$ are the complement of the points that intersect with $x-l$. As $l$ has 11 points the complement of an even number of points must be an odd number of points. Note that if $x$ is the sum of an even number of lines a similar argument shows that the number of points in the intersection of $x$ and $l$ is even. \end{proof}

\begin{lemma} The weights of code words in the code generated by the incidence matrix for the projective plane of order 10 are 0 or 3  (mod 4) for even or odd numbers of summed lines, respectively.\end{lemma}

\begin{proof} Start with the base case of the sum of zero lines from the incidence matrix. This is the zero code word and has weight $0$ (mod $4$). Any single row of the incidence matrix as a code word has weight $11=3$ (mod $4$). Now consider adding a line $l$, to a code word $x$, that is the sum of an even number of lines of the incidence matrix. The additional line will intersect the code word an even number of times, $2p$. Then the weight of the resulting code word has
$$w(x)+w(l)-2\cdot2p=0+3-4p \: \textnormal{(mod 4)}=3\:  \textnormal{(mod 4).}$$
If adding a line $l$ to a code word $x$ that is the sum of an odd number of lines, then the line will intersect the code word in an odd number of points, $2p+1$. Then the weight of the new code word is
$$w(x)+w(l)-2(2p+1)=3+3-4p-2\: \textnormal{(mod 4)}= 0\: \textnormal{(mod 4).}$$ \end{proof}

This allows us to express the weight enumerator more neatly since we now have $A_{4i+1}=A_{4i+2}=0$ for all $i$. Therefore we get separate sums for configurations of even weights and odd weights:
$$W_C(x,y)=\sum_{i=0}^{27}A_{4i}x^{111-4i}y^{4i}+\sum_{i=0}^{27}A_{4i+3}x^{108-4i}y^{4i+3}.$$

\begin{definition}
For a code $C$ the orthogonal dual code $C^\perp$ is a code of the same length as $C$ defined as  
$$C^\perp=\{x\mid\forall y \in C, \:y\cdot x=0\: \textnormal{(mod 2)}\}$$
\end{definition}

For finite dimensional vector spaces the orthogonal subspace has complementary dimension to the original subspace. That is, where $C$ is code space of length $N$
$$\dim(C)+\dim(C^\perp)=N.$$

\begin{example} Let $C$ be any code of length n. The zero code word $\underline{0}=(0,\dots,0)$ is in the orthogonal dual code as for any $c=(c_1,\dots,c_n)\in C$ we get $$c\cdot \underline{0}=\sum_{i=1}^n c_i\cdot 0=0$$
\end{example}

\begin{lemma} If $C$ is the code for the projective plane of order 10 and $x\in C$ then $x\in C^\perp$ if and only if $x$ is the sum of an even number of lines.\end{lemma}

\begin{proof} Let $x\in C$ be the sum of an even number of lines and let $y$ be any code word from $C$ which we can write as the sum of lines, $y=l_1+\cdots+l_k$. Any $l_i$ from $y$ intersects $x$ in an even number of points by Lemma 13 which means that there are an even number of places in the code where both $l_i$ and $x$ have one as an entry. Then $l_i\cdot x$ is the sum of an even number of ones so $l_i\cdot x=0$. By the distributive property of inner product we have
$$x \cdot y = x \cdot (l_1+\dots+l_k) = x\cdot l_1+\dots+ x\cdot l_k = 0+\dots+0=0$$
so $x\in C^\perp$.
If $x$ is the sum of an odd number of lines then take any line, $l\in C$ and the number of intersection points of $l$ and $x$ is odd. This means that the number of columns where $l$ and $x$ both have the entry one is odd and $l\cdot x = 1$. Therefore $x\notin C^\perp$. \end{proof}

\begin{lemma} If $C$ is the binary code for a finite projective plane then $$\dim(C\cap C^\perp)=\dim(C)-1.$$
\end{lemma}

\begin{proof}Lemma 17 tells us that where $x\in C$ is the sum of an odd number of lines $x\notin C^\perp$. This tells us
$$\dim(C\cap C^\perp)\leq\dim(C^\perp)<\dim(C). $$
Let us now assume that $\dim(C)=k$ so we can choose a set of $k$ vectors of the code that form a basis for $C$. The lines of the plane are the generators for the code so they must be able to provide such a basis which we will write as $\{l_1,l_2,\ldots,l_k\}$. Next consider the vectors generated by $\{l_2+l_1,l_3+l_1,\ldots,l_k+l_1\}$, a linearly independent set of $k-1$ code words. These code words are a basis for a vector subspace $L\subset C$ as the dimension of $L$ is less than that of $C$. Additionally, since the generators for $L$ are all the sum of an even number of lines, every code word of $L$ is the sum of an even number of vectors. As a result, $L\subseteq C^\perp$ by Lemma 17 so $L\subseteq C\cap C^\perp$ resulting in
$$k-1=\dim(L)\leq \dim(C\cap C^\perp)<\dim(C)=k.$$
As the dimension can only take integer values we know
$$\dim(C\cap C^\perp)=k-1=\dim(C)-1.$$ \end{proof}

\begin{lemma} For $n$ divisible by $2$ exactly once, the code space $C$ of a projective plane of order $n$ has the properties:
\begin{enumerate}
\item $C^\perp \subset C$
\item $\dim(C)=\dfrac{n^2+n+2}{2}$ \quad and \quad $\dim(C^\perp)=\dfrac{n^2+n}{2}.$ 
\end{enumerate}
\end{lemma}

\begin{proof}For an incidence matrix of the projective plane $A$ recall that $A$ has dimensions $(n^2+n+1)\times(n^2+n+1)$. To simplify notation let $N=n^2+n+1$. Now we perform determinant preserving operations on $A$ since we previously calculated the determinant:
$$\det(A)=\pm (n+1)n^{(n^2+n)/2}.$$
First, we add each of the first $N-1$ columns to the last column resulting in
$$\det(A)=
\begin{vmatrix}
    B_1 & \begin{matrix} n+1 \\ n+1 \\ \vdots  \end{matrix} \\
    \begin{matrix} * & \cdots & * \end{matrix} & n+1
\end{vmatrix}.$$
Then we subtract the bottom row from each of the others giving
$$\det(A)=
\begin{vmatrix}
    B_2 & \begin{matrix} 0 \\ 0 \\ \vdots \\ 0 \end{matrix} \\
    \begin{matrix} * & \cdots & * \end{matrix} & n+1
\end{vmatrix}=(n+1)\det(B_2)$$
The last equality derives from using cofactor expansion on the last column of the matrix.
The determinant of $A$ is non-zero so we can infer that the determinant of $B_2$ is also non-zero and therefore $B_2$ can be diagonalised by row operations which do not change the absolute value of the determinant.
$$\det(A)=\pm\begin{vmatrix}
    \begin{matrix} h_1 & & & 0 \\ & h_2 & & \\ & & \ddots & \\ 0 & & &h_{N-1} \end{matrix} & \begin{matrix} 0 \\ 0 \\ \vdots \\ 0 \end{matrix} \\
    \begin{matrix} * & \cdots & \cdots & * \end{matrix} & n+1
\end{vmatrix}=\pm |A_2|$$
As this matrix is lower triangular we can find the determinant by taking the product of the diagonal elements.
\begin{align*}\det(A)=\pm h_1h_2\cdots h_{N-1}(n+1)&= \pm (n+1)n^{(n^2+n)/2} \\
h_1h_2\cdots h_{n^2+n}&=\pm n^{(n^2+n)/2}
\end{align*}
As we have $n$ divisible by 2 only once then $(n^2+n)/2$ is the maximum number of $h_i$ that are divisible by 2. Now we consider a representation of $A_2$ over a binary field by taking each entry modulo 2. The $h_i$ which were divisible by 2 are now zero in the matrix and so the corresponding rows are zero rows and dependent. Each of the other rows (of which there are at least $N-(n^2+n)/2=(n^2+n+2)/2$) contains a one on the diagonal (noting that $n+1=1$ (mod 2) as n is divisible by 2) and therefore they are linearly independent. Then
$$\dim(C)\geq (n^2+n+2)/2$$
and by Lemma 18
$$\dim(C^\perp)\geq \dim(C\cap C^\perp)=\dim(C)-1 \geq (n^2+n)/2.$$
As $\dim(C)$ and $\dim(C^\perp)$ must sum to the length of the code $N=n^2+n+1$ the preceding equations must be equalities.
The first property follows from $\dim(C\cap C^\perp)=\dim(C)-1=\dim(C^\perp)$ given by Lemma 18. Since the intersection is a subset of $C^\perp$ with the dimension of $C^\perp$ the intersection is the entirety of $C^\perp$. Therefore $C^\perp \subset C.$\end{proof}

%%%%%%%%%%%%%%%%%%%%%%%%%%%%%%%%%%%%%%%%%%%%%%%%%%%%%%%%%%%%%%%%%%%%%%%%%%%%%%%%%%%%%%%%%%%%%%%%%%%%%%%%%%%%%%%%%%%%%%%%%%%%
\subsection{Weight-enumerator equating}

Finding the weight-enumerator polynomial for the projective plane of order 10 will tell us about its structure. It reflects the number of code words of each weight in the code which should be a natural number, so a negative or non-integer coefficient for any term of the polynomial would indicate the plane could not exist. A theorem given here without proof will be immensely useful. \\

\begin{theorem} (MacWilliams) For a binary code C and its orthogonal dual code $C^\perp$ the weight enumerators have the property
$$W_{C^\perp}(x,y)=\frac{1}{|C|}W_C(x+y,x-y).$$
\end{theorem}

Along with the MacWilliams identity we need to try and find as many $A_i$ values that are easily determined. Firstly, we know from Lemma 14 that there are no code words of weight 1 or 2 (mod 4). We also know that $A_0=1$ because the zero code word is the unique weight zero configuration in the code. \\

\begin{lemma} For the code of the projective plane of order 10, there are no configurations containing $k$ points for $1\leq k \leq 10$ so $A_1 = A_2 = \cdots = A_{10} = 0$. \end{lemma}

\begin{proof} Let $x\in C$ be a configuration of odd weight less than 111. Then consider a point $p$ not on $x$. There are 11 lines passing through $p$ which must all intersect $x$ in an odd number of points by Lemma 13. Therefore each line must intersect $x$ at least once and the points must all be distinct as the lines have their only intersection at $p$. Therefore $w(x)\geq 11$.\\

If instead $x\in C$ is a configuration of even weight greater than 0, consider a point $p$ on $x$. There are 11 lines on $p$ and each must intersect $x$ in an even number of points (Lemma 13) and therefore at least twice. Each line intersects $x$ at a point other than $p$ and these 11 points are all distinct since the lines intersect at $p$. Therefore $w(x)\geq 12$. \end{proof} 

\begin{lemma} The projective plane of order 10 has 111 configurations of weight 11 so $A_{11}=111.$
\end{lemma}

\begin{proof} There are 111 lines of the projective plane which each contain 11 points. If we show that every code word with weight 11 is a line of the plane then we know they are equivalent and $A_{11}=111$. Consider $x\in C$ where $w(x)=11$. Then $x$ is the sum of an odd number of lines (Lemma 14) so every line of the plane intersects it an odd number of times (Lemma 13). Choose points $a$, $b$ on $x$ which must lie on a unique line. Now we have two possibilities; There is some point $p$ on the line and not on $x$ or all the points of $x$ lie on this line. In the first case consider the 10 other lines passing through $p$ which must each intersect $x$ at least once. All of the points of intersection must be distinct as the lines already have an intersection at $p$. This implies there are 10 more points on $x$ other than $a$ and $b$ which is impossible. Therefore the second case must occur and the line $ab$ contains all points of $x$ so $x$ is a line. \end{proof}

This gives us $A_0=1$, $A_1 = A_2 = \cdots = A_{10} = 0$ and $A_{11}=111$. Finally, note that the values come in pairs where $A_i = A_{111-i}$ as the sum of all the lines contains every point. This means that for all sum of $k$ lines containing $m$ points the other $111-k$ lines sum to the other $111-m$ points and so every code word has a matching configuration of opposite weight.

From Lemma 19 we know that the dimension of $C$ is 56. Being a binary code we can deduce $|C|=2^{56}$. Now we can use the MacWilliams identity to equate:

\hspace*{-1cm}\vbox{\begin{align*}\sum_{i=0}^{27}A_{4i}x^{111-4i}y^{4i}= \frac{1}{|C|}\left[ \sum_{i=0}^{27}A_{4i}(x+y)^{111-4i}(x-y)^{4i}+\sum_{i=0}^{27}A_{4i+3}(x+y)^{108-4i}(x-y)^{4i+3}\right]\end{align*}}

Equating the coefficients for the polynomials gives a system of linear equations in terms of $A_{4i}$ as $A_{4i+3}=A_{111-(4i+3)}=A_{4(27-i)}$. Reducing in Mathematica tells us that the system is underdetermined and there are three degrees of freedom. The workbook used in equating the polynomials can be found at 	
goo.gl/Gl0gCi. As shown by the output three more $A_{i}$ values need to be calculated to find the weight enumerator. Computer searches provided answers, finding $A_{12}=A_{15}=A_{19}=0$ \cite{B10,B15C,B12}. The method used for the searches will be explained in more depth in Section 4.

Now all the $A_i$ are determined and are given here:
$$\begin{array}{lrr}
\hline
\multicolumn{2}{c}{i}  & \multicolumn{1}{c}{A_i}\\
\hline
0 & 111 & 1 \\
11 & 100 & 111 \\
19 & 92 & 24,675 \\
20 & 91 & 386,010 \\
23 & 88 & 18,864,495 \\
24 & 87 & 78,227,415 \\
27 & 84 & 2,698,398,790 \\
28 & 83 & 8,148,873,195 \\
31 & 80 & 166,383,964,620 \\
32 & 79 & 415,533,405,150 \\
35 & 76 & 5,023,148,053,500 \\
36 & 75 & 10,604,483,511,375 \\
39 & 72 & 78,347,862,432,300 \\
40 & 71 & 141,031,595,676,060 \\
43 & 68 & 653,162,390,747,370 \\
44 & 67 & 1,009,413,831,402,540 \\
47& 64 & 2,982,186,455,878,665 \\
48& 63 & 3,976,279,652,851,020 \\
51& 60 & 7,582,305,834,092,682 \\
52& 59 & 8,748,789,607,170,360 \\
55& 56 & 10,841,059,295,003,634 \\
\hline
\end{array}$$

\subsection{Code word of weight 19}
Now that the weight enumerator has been determined it can be seen that $A_{19} = 24,675$ is the smallest configuration which is the sum of more than one line that is present in the code. The fact that a weight 19 configuration exists in the code provides information about the possible structure of the incidence matrix so it can be used as a basis for generating the matrix.

Let $x\in C$ be a code word of weight 19. As $x$ contains an odd number of points it is the sum of an odd number of lines and therefore any line, $l$, intersects $x$ in an odd number of points. More specifically, $l$ must intersect $x$ in 1, 3 or 5 points. This is because if there are 11, 9 or 7 points of intersection for $x$ and $l$ then $w(x+l)=8$, 12 or 16 respectively. We already know there are no configurations of these weights from the weight enumerator so there must be 1, 3 or 5 intersections from lines of the plane and we will call them ``single lines'', ``triple lines'' and ``heavy lines'' correspondingly.

Now we make a system of equations to find the number of lines of the plane intersecting $x$ in 1, 3 or 5 points which we will represent by $l_1$, $l_3$ and $l_5$. There are 111 lines of the plane in total and they are partitioned by how many times they intersect $x$ giving
$$ l_1+l_3+l_5=111.$$
Another equation is formed by considering how many lines need to lie on the points of $x$. Each of the 19 points in $x$ is on 11 lines so collectively the lines need to intersect $x$ $19\times 11 = 209$ times. Since each $l_1$ intersects $x$ once; $l_3$, three times and the $l_5$ intersect $x$ five times we have
$$l_1+3l_3+5l_5=209.$$
Lastly, we expect any two points of $x$ to lie on exactly one line. The number of distinct pairs of points in $x$ is $\binom{19}{2}=171$. No single line contains two points of $x$ but each triple line contains three points of $x$ so provides a line for $\binom{3}{2}=3$ distinct pairs. Similarly each heavy line contains five points of $x$ so contains $\binom{5}{2}=10$ pairs in $x$. Therefore our third equation is
$$3l_3+10l_5=171.$$
Solving the system gives us these solutions:
$$ l_1=68 \quad l_3=37 \quad  l_5=6$$

Next we can consider how the incidence matrix might look. We can write it as

$$\begin{blockarray}{ccc}
\begin{block}{ccc}
& 19 & 92 \\
\end{block}
\begin{block}{c[cc]}
6 & A_1 & A_2 \\
37 & B_1 & B_2 \\
68 & C_1 & C_2 \\
\end{block}
\end{blockarray}$$

where the first 19 columns ($A_1,B_1,C_1$) of the matrix represent the points in $x$. In the matrix $A_1$ and $A_2$ are the 6 heavy lines; $B_1$ and $B_2$ are the 37 triple lines and $C_1$ and $C_2$ are the 68 single lines. 

Now we want to know more about the structure of the lines passing through $x$. Consider a point $p\in x$ and let $p_1$, $p_3$ and $p_5$ be the number of single, triple and heavy lines containing $p$. Then as $p$ is on 11 lines we conclude
$$p_1+p_3+p_5=11.$$
The other 18 points of $x$ must each lie on precisely one line through $p$. Each triple line passes through two other points of $x$ and each heavy line passes through four other points of $x$ so we get the equation
$$2p_3+4p_5=18.$$
Choosing integer values for $p_5$ allows us to generate the following table of possibilities.
$$\begin{array}{ccc}
p_1 & p_3 & p_5\\
\hline
2 & 9& 0\\
3 & 7 & 1 \\
4 & 5 & 2 \\
5 & 3 & 3 \\
6 & 1 & 4
\end{array}
$$

Now we can consider options for $A_1$. This is the incidence matrix where the heavy lines intersect $x$. There are 66 non-isomorphic matrices meaning the starting points differ in structure so they cannot be made to look the same by simply changing the order of rows or columns. Two examples of non-isomorphic configurations for $A_1$ are:

$$X =\begin{array}{*{22}c}
1 & 1 & 1 & 1 & 1 & & 0 & 0 & 0 & 0 & 0 & & 0 & 0 & 0 & 0 & 0 & & 0 & 0 & 0 & 0 \\
1 & 0 & 0 & 0 & 0 & & 1 & 1 & 1 & 1 & 0 & & 0 & 0 & 0 & 0 & 0 & & 0 & 0 & 0 & 0 \\
1 & 0 & 0 & 0 & 0 & & 0 & 0 & 0 & 0 & 1 & & 1 & 1 & 1 & 0 & 0 & & 0 & 0 & 0 & 0 \\
1 & 0 & 0 & 0 & 0 & & 0 & 0 & 0 & 0 & 0 & & 0 & 0 & 0 & 1 & 1 & & 1 & 1 & 0 & 0 \\
0 & 1 & 0 & 0 & 0 & & 1 & 0 & 0 & 0 & 1 & & 0 & 0 & 0 & 1 & 0 & & 0 & 0 & 1 & 0 \\
0 & 0 & 1 & 0 & 0 & & 0 & 1 & 0 & 0 & 0 & & 1 & 0 & 0 & 0 & 0 & & 0 & 0 & 1 & 1 \\
\end{array}$$

$$Y =\begin{array}{*{22}c}
1 & 1 & 1 & 1 & 1 & & 0 & 0 & 0 & 0 & 0 & & 0 & 0 & 0 & 0 & 0 & & 0 & 0 & 0 & 0 \\
1 & 0 & 0 & 0 & 0 & & 1 & 1 & 1 & 1 & 0 & & 0 & 0 & 0 & 0 & 0 & & 0 & 0 & 0 & 0 \\
1 & 0 & 0 & 0 & 0 & & 0 & 0 & 0 & 0 & 1 & & 1 & 1 & 1 & 0 & 0 & & 0 & 0 & 0 & 0 \\
0 & 1 & 0 & 0 & 0 & & 1 & 0 & 0 & 0 & 1 & & 0 & 0 & 0 & 1 & 1 & & 0 & 0 & 0 & 0 \\
0 & 0 & 1 & 0 & 0 & & 0 & 1 & 0 & 0 & 0 & & 1 & 0 & 0 & 0 & 0 & & 1 & 1 & 0 & 0 \\
0 & 0 & 0 & 1 & 0 & & 0 & 0 & 1 & 0 & 0 & & 0 & 1 & 0 & 0 & 0 & & 0 & 0 & 1 & 1 \\
\end{array}$$

Both configurations look like they have potential to be part of our incidence matrix. However, it can be shown that neither of these are viable as $A_1$.

In $X$ we can see there are four ones in the first column so this point of $x$ has four heavy lines on it. From the table we expect it to also have one triple line on it. This triple line must have points 1, 18 and 19 on it since every point other than 18 and 19 is already on a line through 1. The 6th row already has 18 and 19 on it which is a contradiction since there should only be one line containing both. Therefore $X$ cannot be $A_1$ for the incidence matrix of a projective plane of order 10.

Starting with $Y$ also leads to contradictions. Notice that the last three rows of $Y$ have no intersection within $x$. They must intersect somewhere outside $x$, but not in a single point $p$ or this point lies on three heavy lines and has eight other lines which intersect $x$ at least once. This would mean the lines through $p$ would intersect $x$ in at least 23 distinct points, clearly impossible for a code word of weight 19. Consider the code word generated by adding $x$ and the last three lines. Since none of the lines intersect in $x$, a point in $x$ is either on one line or none of them. The points of $x$ that are in the resulting code word are 1, 5, 9 and 13. Points outside $x$ that are in the code word are points on one of the heavy lines and not at the intersection of two of them. Each of the three heavy lines contains six points outside of $x$, 4 distinct points and two points providing intersections with the other two heavy lines. Therefore the code word has 1, 5, 9, and 13 from $x$ and four points from outside of $x$ for each heavy line for a total of 16 points. As this was already excluded as a possibility by the weight enumerator $Y$ is not a possible start for the incidence matrix.  

Of the 66 possibilities for $A_1$ similar arguments rule out 21 of the  $A_1$ configurations. This leaves 45 starting submatrices of size $6\times 19$ that have the potential to be extended. The process undertaken by computer search was first to extend the $6\times 19$ submatrix to one of size $43\times 19$ by finding all possible ways the triple lines could lie on the 19 points of $x$. Each of these has a unique extension (up to isomorphism) to a $111\times 19$ configuration by adding the incidence of the single lines on $x$.

Next the incidence of points not in $x$ needed to be found. Each heavy line contains six points outside of $x$ so the six columns corresponding to one of the heavy lines is completed first. After this the external points of another heavy line are chosen to continue the matrix. This next addition might only consist of five new columns if the heavy lines intersect outside $x$ and this makes the search more efficient where possible. When the computer search was carried out, although every configuration could be extended to some incidence submatrix on the points of 4 heavy lines, none of the 45 starting arrangements could generate incidence on a 5th heavy line that was consistent with the rest of the matrix. Therefore it was concluded that no incidence matrix exists for the projective plane of order 10 and that the plane too cannot exist.

%%%%%%%%%%%%%%%%%%%%%%%%%%%%%%%%%%%%%%%%%%%%%%%%%%%%%%%%%%%%%%%%%%%%%%%%%%%%%%%%%%%%%%%%%%%%%%%%%%%%%%%%%%%%%%%%%%%%%%%%%%%%%%%%%%%%%%%
%%%%%%%%%%%%%%%%%%%%%%%%%%%%%%%%%%%%%%%%%%%%%%%%%%%%%%%%%%%%%%%%%%%%%%%%%%%%%%%%%%%%%%%%%%%%%%%%%%%%%%%%%%%%%%%%%%%%%%%%%%%%%%%%%%%%%%%
\newpage
\section{A Search for Code Words of Weight 15}

In this section we will explore how MacWilliams, Thompson and Sloane \cite{B15C} demonstrated that there are no code words, $x$, of weight 15 generated by the incidence matrix of the projective plane of order 10. We will follow the reasoning of their article and carry out a search for code words of weight 15 in order to confirm their result. The method used is similar to that shown in the previous section where the existence of such a configuration $x$ generates a submatrix which cannot be extended to an entire incidence matrix.

The Lemmas from section 3 provide almost enough background for us to start looking at properties of a projective plane where there is a 15 point configuration. Beforehand however, we need to talk about properties of hyperovals of the plane.

\subsection{Hyperovals of the Plane}

\begin{definition}A hyperoval of a projective plane of order $n$ where $n$ is even is a set of $n+2$ points, no three of which are collinear. \\
\end{definition}

\begin{lemma} Code words of weight 12 in the code are precisely the hyperovals of the plane.
\end{lemma}

\begin{proof} Let $x$ be a code word of weight 12. Every line of the plane intersects $x$ in an even number of points by Lemma 13. Consider any point $p$ on $x$. Then $p$ has 11 lines passing through it and each line must intersect $x$ in at least one other point to have an even number of intersections. Each of these additional points must be distinct since the lines intersect at $p$ so the 11 points and $p$ account for all 12 points and no line can intersect more than twice. Therefore $x$ is a hyperoval.

Now take a hyperoval, $\sigma$. Let $L$ be the set of $\binom{12}{2}=66$ lines which intersect with $\sigma$. Consider the submatrix of the incidence matrix corresponding to the lines of $L$.

$$\begin{blockarray}{ccccccccc}
\begin{block}{cc\BAmulticolumn{4}{c}ccc}
& & \sigma & & & \\
\end{block}
\begin{block}{cc(cccc)ccc}
& & 1 & 2 & \cdots & 12 & 13 & \cdots & 111 \\
\end{block}
\begin{block}{cc[ccccccc]}
& 1 & a_{1,1} & a_{1,2} & \cdots & a_{1,12} & a_{1,13} & \cdots & a_{1,111} \\
& 2 & a_{2,1} & a_{2,2} & \cdots & a_{2,12} & a_{2,13} &\cdots & a_{2,111} \\
L & 3 & a_{3,1} & a_{3,2} & \cdots & a_{3,12} & a_{3,13} & \cdots & a_{3,111} \\
 & \vdots & \vdots &  \vdots & \ddots & \vdots & \vdots & \ddots & \vdots \\
 & 66 & a_{66,1} & a_{66,2} & \cdots & a_{66,12} & a_{66,13} & \cdots & a_{66,111} \\
\end{block}
\end{blockarray}$$

The first 12 columns each contain 11 ones since the 11 lines lying on each point of $\sigma$ must be part of the line set $L$. Each of the 66 rows  also contains 11 ones as they are entire lines of the projective plane containing 11 points. This leaves $66\cdot 11-12\cdot 11=54\cdot 11=6\cdot 99$ ones in columns 13 onward. No point $p\notin \sigma$ lies on more than 6 lines as that implies there are 7 lines passing through $p$ that intersect $\sigma$. Each line must intersect $\sigma$ twice and the points of intersection must be distinct as the lines all pass through $p$ and this is not possible as $\sigma$ does not contain 14 points. As there are 99 points not in $\sigma$ and $6\cdot 99$ points to distribute among them with no more than 6 in each then there must be precisely 6 ones in each of the columns 13-111.
Now if we take the linear combination that is the sum of all 66 of these lines in binary then $p\in \sigma$ lies on an odd number (11) and $p\notin \sigma$ lies on an even number of lines (6). The resulting configuration is a vector of weight 12 in the code space. \end{proof}

\subsection{Assuming a Weight 15 Code Word Exists}
We will assume there is some code word of weight 15 and see what is implied.
Let $A=\{1,2,\ldots,15\}$ be the set of points contained in such a code word. \\

\begin{lemma}Every line of the projective plane intersects A in 1, 3 or 5 points.
\end{lemma}
\textbf{Proof.} Let $l$ be any line of the projective plane. It must intersect $A$ in an odd number of points by Lemma 13. Then by Lemma 10 we get
\begin{align*} w(A+l)&=w(A)+w(l)-2w(A\cap l) \\
&=15+11-2w(A\cap l)
\end{align*}
Lemma 21 tells us that weight of $A+l$ cannot be 8 or 4 so $w(A\cap l)$ cannot be 9 or 11. Lastly, the number of intersections of $A$ and $l$ cannot be 7 or $A+l$ contains 12 points and is therefore a hyperoval, yet $l$ intersects $A+l$ in the 4 points of $l$ not in $A$. This contradicts our earlier definition of hyperovals. Therefore $A$ and $l$ intersect in 1, 3 or 5 points. \\

\begin{lemma}Of the 111 lines in the projective plane, 90 intersect $A$ once, 15 intersect $A$ three times and the remaining 6 intersect five times. 
\end{lemma}

\begin{proof} We will use the same notation and technique as we used when looking at the weight 19 code word in the last section. This means we will give the names single lines, triple lines and heavy lines to lines with 1,3 or 5 intersections with $A$ and let the number of each be $l_1,l_3$ and $l_5$ respectively.

Every line of the projective plane intersects $A$ in one of 1,3 or 5 points so one equation is
$$l_1+l_3+l_5=111.$$
We can also consider the fact that every pair of points in $A$ (of which there are $\binom{15}{2}=105$) must have a line passing through both. Each triple line passes through $\binom{3}{2}=3$ pairs of points of $A$ and each heavy line passes through $\binom{5}{2}=10$ pairs giving
$$3l_3+10l_5=105.$$
A third equation is generated by noting that each point of $A$ has 11 lines passing through it so the total number of intersection of lines with points in $A$ is $11\cdot15$. We also know how many times each single, triple and heavy line intersects points of $A$ so 
$$ l_1+3l_3+5l_5=11\cdot 15.$$
Solving the 3 equations gives us the result
$$ l_1=90 \quad \quad l_3=15 \quad \quad l_5=6. $$ \end{proof}

Using notation consistent with \cite{B15C} we label the heavy lines $B_1,\ldots,B_6$.\\

\begin{lemma}The intersection of two distinct $B_i$s is a point in $A$.
\end{lemma}
\begin{proof} If there is a point $p\notin A$ where two $B_i$s intersect then $p$ has 9 other lines passing through it. Each of these 9 lines intersects $A$ at least once by Lemma 13, but each of the intersections should be distinct from each other and also the 10 points of the $B_i$s. This is not possible as $A$ only has 15 points. Therefore the $B_i$s must intersect in an element of $A$. \end{proof}

\begin{lemma} No point lies on more than two $B_i$s
\end{lemma}

\begin{proof} Let point $p\in A$ lie on lines $B_1$, $B_2$ and $B_3$. It is not possible for another heavy line to contain $p$ as there are only two remaining points of $A$. Let $B_4$ contain one point of each of $B_1$,$B_2$ and $B_3$ and both of the two remaining points of $A$. Then $B_5$ can only contain a maximum of one point from each of the previous lines and then can only have four points of intersection with $A$ which is a contradiction. \end{proof}

Now we know that each pair of $B_i$s intersects in a point of $A$, but no three intersect at a single point. There are $\binom{6}{2}=15$ pairs of $B_i$s so the distinct intersection points mean that each point of $A$ is the intersection of two heavy lines. Without loss of generality we can now write a matrix representing the points of $A$ on each $B_1,\ldots,B_6$.

$$\begin{matrix}
B_1: & 1 & 2 & 3 & 4 & 5 \\ 
B_2: & 1 & 6 & 7 & 8 & 9 \\ 
B_3: & 2 & 6 & 10 & 11 & 12 \\ 
B_4: & 3 & 7 & 10 & 13 & 14 \\ 
B_5: & 4 & 8 & 11 & 13 & 15 \\ 
B_6: & 5 & 9 & 12 & 14 & 15 \\ 
\end{matrix}$$

As the intersections of these lines is in $A$ all the remaining points on $B_i$ must be distinct. We let B=\{76,\ldots,111\} be the 36 other points on the $B_i$s.
Now we want to move on to looking at how the triple lines intersect $A$. Let $C_1,\ldots,C_{15}$ be the lines that intersect $A$ in 3 points. \\

\begin{lemma}Each point of $A$ lies on exactly three $C_i$s
\end{lemma}

\begin{proof} Each $p\in A$ is on two heavy lines containing four other points of $A$. This only leaves six other points of $A$ for $C_i$s on $p$ to pass through so there must be a maximum of 3. In total the 15 $C_i$s intersect $A$ 45 times (3 each) so this is an average of 3 $C_i$s through each point of $A$. Therefore, as no point lies on more than 3, each point of $A$ is incident with precisely 3 $C_i$s. \end{proof}

Now we can form a matrix representing the incidence of the $C_i$s on $A$ based on the $B_i$ matrix we already have. Up to permutation we must have the following arrangement.

$$\begin{matrix}
C_1: & 1 & 10 & 15 \\ 
C_2: & 1 & 11 & 14 \\
C_3: & 1 & 12 & 13 \\
C_4: & 2 & 7 & 15 \\
C_5: & 2 & 8 & 14 \\ 
C_6: & 2 & 9 & 13 \\ 
C_7: & 3 & 6 & 15 \\
C_8: & 3 & 8 & 12 \\
C_9: & 3 & 9 & 11 \\
C_1: & 4 & 6 & 14 \\
C_1: & 4 & 7 & 12 \\
C_1: & 4 & 9 & 10 \\
C_1: & 5 & 6 & 13 \\
C_1: & 5 & 7 & 11 \\
C_1: & 5 & 8 & 10 \\
\end{matrix}$$

Every $C_i$ intersects with each $B_i$ in a point in $A$. Therefore there are 60 points remaining with which to finish the entries of the $C_i$ rows. These points will be represented by the set $C=\{16,\ldots,75\}$. \\

\begin{lemma} Each point in C is the intersection of exactly two $C_i$
\end{lemma}
\begin{proof} If a point $p\in C$ is on three $C_i$s then these $C_i$s already intersect 9 points of $A$. The remaining 8 lines containing $p$ should each have at least one intersection with $A$, but there are only 6 points of $A$ left which is a contradiction. Therefore each point of $C$ can be contained in two $C_i$s at most.
Each $C_i$ contains 8 points from $C$. As there are 15 of these lines there are 120 points needed to fill in the incidence table. There are only 60 points in C, none of which can be in more than 2 rows so each must be used exactly twice to create the $C_i$ lines. \end{proof}

Now we can complete the $B_i$ and $C_i$ rows.

$$\begin{array}{*{12}c}
1 & 2 & 3 & 4 & 5 & & 76 & 77 & 78 & 79 & 80 & 81 \\
1 & 6 & 7 & 8 & 9 & & 82 & 83 & 84 & 85 & 86 & 87 \\
2 & 6 & 10 & 11 & 12 & & 88 & 89 & 90 & 91 & 92 & 93 \\
3 & 7 & 10 & 13 & 14 & & 94 & 95 & 96 & 97 & 98 & 99 \\
4 & 8 & 11 & 13 & 15 & & 100 & 101 & 102 & 103 & 104 & 105 \\
5 & 9 & 12 & 14 & 15 & & 106 & 107 & 108 & 109 & 110 & 111 \\
\hline
1 & 10 & 15 & & 16 & 17 & 18 & 19 & 20 & 21 & 22 & 23 \\ 
1 & 11 & 14 & & 24 & 25 & 26 & 27 & 28 & 29 & 30 & 31 \\
1 & 12 & 13 & & 32 & 33 & 34 & 35 & 36 & 37 & 38 & 39 \\
2 & 7 & 15 & & 24 & 32 & 40 & 41 & 42 & 43 & 44 & 45 \\
2 & 8 & 14 & & 16 & 33 & 46 & 47 & 48 & 49 & 50 & 51 \\ 
2 & 9 & 13 & & 17 & 25 & 52 & 53 & 54 & 55 & 56 & 57 \\ 
3 & 6 & 15 & & 26 & 34 & 46 & 52 & 58 & 59 & 60 & 61 \\
3 & 8 & 12 & & 18 & 27 & 40 & 53 & 62 & 63 & 64 & 65 \\
3 & 9 & 11 & & 19 & 35 & 41 & 47 & 66 & 67 & 68 & 69 \\
4 & 6 & 14 & & 20 & 36 & 42 & 54 & 62 & 66 & 70 & 71 \\
4 & 7 & 12 & & 21 & 28 & 48 & 55 & 58 & 67 & 72 & 73 \\
4 & 9 & 10 & & 29 & 37 & 43 & 49 & 59 & 63 & 74 & 75 \\
5 & 6 & 13 & & 22 & 30 & 44 & 50 & 64 & 68 & 72 & 74 \\
5 & 7 & 11 & & 23 & 38 & 51 & 56 & 60 & 65 & 70 & 75 \\
5 & 8 & 10 & & 31 & 39 & 45 & 57 & 61 & 69 & 71 & 73 \\
\end{array}$$

\begin{lemma}The 90 single lines of the projective plane each contain four points of B and six points of C.
\end{lemma}
\begin{proof} Let $l$ be a single line. Then $l$ intersects $A$ at one point, which will be the intersection of two $B_i$s. As $l$ contains no other points of $A$ it must contain four points of $B$ in order to intersect with the other four $B_i$s which have distinct points of $B$. The total number of points on $l$ is 11 so the remaining six points are in $C$. \end{proof}

\subsection{Group Operations on A}
Next we want to examine the structure of our current arrangement to reduce the computation we need to carry out. Let points of $A$ correspond to transpositions of the symbols $t_1,\ldots,t_6$ decided by the $B_i$ lines each point of $A$ lies on:
$$\begin{array}{cccccccc}
1&2&3&4&5&6&7&8 \\
(t_1t_2)&(t_1t_3)&(t_1t_4)&(t_1t_5)&(t_1t_6)&(t_2t_3)&(t_2t_4)&(t_2t_5)
\end{array}$$
$$\begin{array}{ccccccc}
9&10&11&12&13&14&15 \\
(t_2t_6)&(t_3t_4)&(t_3t_5)&(t_3t_6)&(t_4t_5)&(t_4t_6)&(t_5t_6)
\end{array}$$
Then the operations $\tau_i$ representing conjugation by $(t_i,t_{i+1})$ for i=1,\ldots,5 act on $A$ and generate a group of operations, $G$. Furthermore, the set $A$ is invariant under $G$ as shown by the following table of $\tau_i$s operating on $A$. 

$$\begin{array}{c|ccccccccccccccc}
 & 1 & 2 & 3 & 4 & 5 & 6 & 7 & 8 & 9 & 10 & 11 & 12 & 13 & 14 & 15 \\
 \hline
 \tau_1 & 1 & 6 & 7 & 8 & 9 & 2 & 3 & 4 & 5 & 10 & 11 & 12 & 13 & 14 & 15 \\
 \tau_2 & 2 & 1 & 3 & 4 & 5 & 6 & 10 & 11 & 12 & 7 & 8 & 9 & 13 & 14 & 15 \\
 \tau_3 & 1 & 3 & 2 & 4 & 5 & 7 & 6 & 8 & 9 & 10 & 13 & 14 & 11 & 12 & 15 \\
 \tau_4 & 1 & 2 & 4 & 3 & 5 & 6 & 8 & 7 & 9 & 11 & 10 & 12 & 13 & 15 & 14 \\
 \tau_5 & 1 & 2 & 3 & 5 & 4 & 6 & 7 & 9 & 8 & 10 & 12 & 11 & 14 & 13 & 15
 \end{array}$$
 
The set of transpositions representing the five points of $A$ on $B_1$ is $$\{(t_1t_2),(t_1t_3),(t_1t_4),(t_1t_5),(t_1t_6)\}$$ and we use it to represent the line $B_1$. The other $B_i$s are also represented as the set of 5 points from $A$ they contain and together the set of 6 $B_i$s is invariant under the operations of $G$.  
The $C_i$ lines are also related by conjugation, but are identified with a set of three transpositions. The line $C_1$ can be viewed as $$\{(t_1t_2), (t_3t_4), (t_5t_6)\}.$$ 
Similarly the set of $C_i$s is invariant under the operations of $G$ and this gives us a way of extending $G$ to operate on the points of $C$. Since every point of $C$ is the intersection of two $C_i$s we consider how $G$ acts on those two lines and find the new intersection point so as to maintain incidence under $G$. \\

\begin{example} Consider the point $16\in C$. This point lies on $C_1$ and $C_5$ which are represented by $\{(t_1t_2), (t_3t_4), (t_5t_6)\}$ and $\{(t_1t_3), (t_2t_5), (t_4t_6)\}$. If we apply $\tau_4$ as an operation on 16 we need to look at how $\tau_4$ affects $C_1$ and $C_5$. 
\begin{align*} \tau_4(C_1)&=\{(t_4t_5)(t_1t_2)(t_4t_5), (t_4t_5)(t_3t_4)(t_4t_5), (t_4t_5)(t_5t_6)(t_4t_5)\} \\
&= \{(t_1t_2), (t_3t_5), (t_4t_6)\} \\
&= \{1,11,14\} \\
&= C_2 \\
\end{align*}
\begin{align*}
\tau_4(C_5)&=\{(t_4t_5)(t_1t_3)(t_4t_5), (t_4t_5)(t_2t_5)(t_4t_5), (t_4t_5)(t_4t_6)(t_4t_5)\} \\
&= \{(t_1t_3), (t_2t_4), (t_5t_6)\} \\
&= \{2,7,15\} \\
&= C_4
\end{align*}
Now we find the point of intersection of $C_2$ and $C_4$ which is $24$, therefore $\tau_4(16)=24$. The table below shows the behaviour of each element of $C$ under each $\tau_i$.
\end{example}

We can reduce computation in a later section by noting that the only $\tau_i$ which does not fix the point 1 is $\tau_2$. Therefore, let $G_1$ be the subgroup generated by $\tau_1$, $\tau_3$, $\tau_4$ and $\tau_5$ under which 1 is unchanged.

$$\begin{array}{c|ccccccccccccccc}
 & 16 & 17 & 18 & 19 & 20 & 21 & 22 & 23 & 24 & 25 & 26 & 27 & 28 & 29 & 30 \\
 \hline
 \tau_1 & 20 & 22 & 21 & 23 & 16 & 18 & 17 & 19 & 26 & 30 & 24 & 28 & 27 & 31 & 25 \\
 \tau_2 & 24 & 32 & 41 & 40 & 42 & 43 & 44 & 45 & 16 & 33 & 46 & 47 & 49 & 48 & 50 \\
 \tau_3 & 18 & 19 & 16 & 17 & 21 & 20 & 23 & 22 & 34 & 35 & 32 & 33 & 36 & 37 & 38 \\
 \tau_4 & 24 & 25 & 28 & 29 & 26 & 27 & 30 & 31 & 16 & 17 & 20 & 21 & 18 & 19 & 22 \\
 \tau_5 & 17 & 16 & 19 & 18 & 22 & 23 & 20 & 21 & 32 & 33 & 34 & 35 & 38 & 39 & 36 \\
 \end{array}$$
 $$\begin{array}{c|ccccccccccccccc}
 & 31 & 32 & 33 & 34 & 35 & 36 & 37 & 38 & 39 & 40 & 41 & 42 & 43 & 44 & 45 \\
 \hline
 \tau_1 & 29 & 34 & 36 & 32 & 38 & 33 & 39 & 35 & 37 & 58 & 60 & 46 & 61 & 52 & 59 \\
 \tau_2 & 51 & 17 & 25 & 52 & 53 & 54 & 55 & 57 & 56 & 19 & 18 & 20 & 21 & 22 & 23 \\
 \tau_3 & 39 & 26 & 27 & 24 & 25 & 28 & 29 & 30 & 31 & 46 & 52 & 58 & 59 & 60 & 61 \\
 \tau_4 & 23 & 33 & 32 & 36 & 37 & 34 & 35 & 39 & 38 & 48 & 49 & 46 & 47 & 50 & 51 \\
 \tau_5 & 37 & 24 & 25 & 26 & 27 & 30 & 31 & 28 & 29 & 41 & 40 & 44 & 45 & 42 & 43 \\ 
 \end{array}$$
 $$\begin{array}{c|ccccccccccccccc}
 & 46 & 47 & 48 & 49 & 50 & 51 & 52 & 53 & 54 & 55 & 56 & 57 & 58 & 59 & 60 \\
 \hline
 \tau_1 & 42 & 70 & 62 & 71 & 54 & 66 & 44 & 72 & 50 & 64 & 68 & 74 & 40 & 45 & 41 \\
 \tau_2 & 26 & 27 & 29 & 28 & 30 & 31 & 34 & 35 & 36 & 37 & 39 & 38 & 59 & 58 & 61 \\
 \tau_3 & 40 & 53 & 62 & 63 & 65 & 64 & 41 & 47 & 67 & 66 & 68 & 69 & 42 & 43 & 44 \\
 \tau_4 & 42 & 43 & 40 & 41 & 44 & 45 & 54 & 55 & 52 & 53 & 57 & 56 & 62 & 66 & 71 \\
 \tau_5 & 52 & 53 & 56 & 57 & 54 & 55 & 46 & 47 & 50 & 51 & 48 & 49 & 60 & 61 & 58 \\ 
 \end{array}$$
 $$\begin{array}{c|ccccccccccccccc}
 & 61 & 62 & 63 & 64 & 65 & 66 & 67 & 68 & 69 & 70 & 71 & 72 & 73 & 74 & 75 \\
 \hline
 \tau_1 & 43 & 48 & 73 & 55 & 67 & 51 & 65 & 56 & 75 & 47 & 49 & 53 & 63 & 57 & 69 \\
 \tau_2 & 60 & 66 & 67 & 68 & 69 & 62 & 63 & 64 & 65 & 71 & 70 & 74 & 75 & 72 & 73 \\
 \tau_3 & 45 & 48 & 49 & 51 & 50 & 55 & 54 & 56 & 57 & 72 & 73 & 70 & 71 & 75 & 74 \\
 \tau_4 & 70 & 58 & 67 & 72 & 73 & 59 & 63 & 74 & 75 & 61 & 60 & 64 & 65 & 68 & 69 \\
 \tau_5 & 59 & 68 & 69 & 66 & 67 & 64 & 65 & 62 & 63 & 72 & 74 & 70 & 75 & 71 & 73 \\ 
 \end{array}$$

\subsection{The 90 Single Lines}

Here we employ the help of a computer in trying to complete the 90 rows corresponding to the single lines of the projective plane. We start by considering the remaining lines passing through the point 1. There are already 2 heavy lines and 3 triples lines that contain 1 so we need to find the 6 single lines that complete the set of lines incident with 1. Each of these lines contains 6 points of $C$ by Lemma 31. The 6 points must be points that are not already on a line containing 1 so they must be selected from the numbers that have not been crossed out in the following table. Additionally these 6 points must provide exactly one intersection with each $C_i$ not passing through 1. This means each row other than the first three must have exactly one of the 6 points on it. One collection of 6 numbers that could lie on a single line has been given in the last column with each number shown beside both rows in which it occurs.

$$\begin{array}{*{12}c|c}
\tikzmark{topA}1 & \tikzmark{topB}{10} & \tikzmark{topC}{15} & & \tikzmark{topD}{16} & \tikzmark{topE}{17} & 18 & 19 & 20 & 21 & 22 & \tikzmark{bottomF}{23} \\ 
\tikzmark{topG}1 & 11 & 14 & & 24 & 25 & 26 & 27 & 28 & 29 & 30 & \tikzmark{bottomG}{31} \\
\tikzmark{topH}1 & 12 & 13 & & 32 & 33 & 34 & 35 & 36 & 37 & 38 & \tikzmark{bottomH}{39} \\
2 & 7 & 15 & & 24 & 32 & 40 & 41 & 42 & 43 & 44 & 45 & 41\\
2 & 8 & 14 & & 16 & 33 & 46 & 47 & 48 & 49 & 50 & 51 & 46 \\ 
2 & 9 & 13 & & 17 & 25 & 52 & 53 & 54 & 55 & 56 & 57 & 53\\ 
3 & 6 & 15 & & 26 & 34 & 46 & 52 & 58 & 59 & 60 & 61 & 46\\
3 & 8 & 12 & & 18 & 27 & 40 & 53 & 62 & 63 & 64 & 65 & 53 \\
3 & 9 & 11 & & 19 & 35 & 41 & 47 & 66 & 67 & 68 & 69 & 41\\
4 & 6 & 14 & & 20 & 36 & 42 & 54 & 62 & 66 & 70 & 71 & 71 \\
4 & 7 & 12 & & 21 & 28 & 48 & 55 & 58 & 67 & 72 & 73 & 72 \\
4 & 9 & 10 & & 29 & 37 & 43 & 49 & 59 & 63 & 74 & 75 & 75 \\
5 & 6 & 13 & & 22 & 30 & 44 & 50 & 64 & 68 & 72 & 74 & 72 \\
5 & 7 & 11 & & 23 & 38 & 51 & 56 & 60 & 65 & 70 & 75 & 75 \\
\tikzmark{bottomA}5 & \tikzmark{bottomB}8 & \tikzmark{bottomC}{10} & & \tikzmark{bottomD}{31} & \tikzmark{bottomE}{39} & 45 & 57 & 61 & 69 & 71 & 73 & 71 \\
\end{array}$$
\DrawLine[red, thick, opacity=0.5]{topA}{bottomA}
\DrawLine[red, thick, opacity=0.5]{topB}{bottomB}
\DrawLine[red, thick, opacity=0.5]{topC}{bottomC}
\DrawLine[red, thick, opacity=0.5]{topD}{bottomD}
\DrawLine[red, thick, opacity=0.5]{topE}{bottomE}
\DrawHLine[red, thick, opacity=0.5]{topA}{bottomF}
\DrawHLine[red, thick, opacity=0.5]{topG}{bottomG}
\DrawHLine[red, thick, opacity=0.5]{topH}{bottomH}

Each set of 6 points fulfilling the above conditions is called a 6-set. There are 344 6-sets for the point 1 and we will let $L_1$ represent the set of all 344 of them. We then generalise this to let $L_i$ denote the set of 6-sets which are being considered as points of the single lines through point $i\in A$. It is possible to generate all the $L_i$s from $L_1$ by using elements of $G$. An element of $G$ that maps 1 to $i\in A$ maps a 6-set of $L_1$s to a 6-set of $L_i$. \\

\begin{example} We can map our sample 6-set for the point 1 to a 6-set for another point by using an element of $G$. We already have a table for how the generators, $\tau_i$ act on points of $C$ so using a transformation of $\tau_2$ will be the easiest to demonstrate. As $\tau_2(1)=2$ we will get a 6-set for $L_2$.
$$\tau_2\{41,46,53,71,72,75\}=\{18,26,35,70,74,73\}$$

\end{example}

We want to choose 6 6-sets from $L_1$ to be elements of $C$ on the lines through 1. As the lines intersect at 1 we need to find all 6 elements subsets $L_1$ where the 6-sets are distinct. Each of these subsets is called a $K_6$ which arises from viewing the elements of $L_1$ as vertices of a graph. Two elements are connected by an edge if the 6-sets corresponding to the vertices do not have any points in common. This means that a set of 6  6-sets looks like a fully connected graph on 6 vertices, a $K_6$, if the 6-sets are all mutually distinct.

We calculated that there are 42,496 possible $K_6$s for the point 1. In order to reduce the number we need to further investigate we consider the permutation $G_1$ and how it acts on these $K_6$s. If a projective plane does exist then it must have these $K_6$s through each point of $A$ that make up the 90 single lines of the plane. Applying a permutation from $G_1$ will map $K_6$s to $K_6$s, however they may now be $K_6$s for a different point of $A$ depending on the action of that element of $G_1$ on the points from $A$ seen on page 27. We picked $G_1$ to preserve the number 1 so we know that the $K_6$ on 1 will remain on the point 1. We may consider these 42,496 $K_6$s through 1 as being separated into 1021 different orbits under $G_1$ and by selecting any one $K_6$ from each orbit we only need to investigate 1021 cases. If any of the $K_6$s through 1 from a particular orbit can be completed into a plane then it means there is a consistent set of 15 $K_6$s through each point of $A$ including this $K_6$. This can be transformed to the representative case from that orbit by an operation of $G_1$ which will preserve the $K_6$s such that there is still a $K_6$ going through each point of $A$. Thus it does not matter which selection we make from the orbit, if we can show that any point of $A$ does not have a viable $K_6$ passing through it then no starting $K_6$ from the orbit is part of a projective plane of order 10.

Each of the 1021 $K_6$s for the lines passing through 1 is considered individually now, so let $U$ be the $K_6$ selected. We have chosen the 6 points from $C$ on each line, but the six lines passing though 1 each still need to contain 4 points from $B$. As $B_1$ and $B_2$ contain the point 1 there are only 24 points of B remaining to be part of the the 6 single lines through 1 and each one of these 24 points must be contained in exactly one line. Each single line must contain one point from each of $B_3$-$B_6$ so as not to intersect more than once with any of the $B_i$ lines and the points from these lines are currently all equivalent so without loss of generality we may assume the arrangement below.

$$\begin{array}{ccccc}
1 & 88 & 94 & 100 & 106 \\
1 & 89 & 95 & 101 & 107 \\
1 & 90 & 96 & 102 & 108 \\
1 & 91 & 97 & 103 & 109 \\
1 & 92 & 98 & 104 & 110 \\
1 & 93 & 99 & 105 & 111
\end{array}$$

Next for any potential set of single lines passing through 1 we can look at the set of 6-sets for another point in $A$ in order to try and find the 6 single lines passing through that point. To reduce calculation times as much as possible we choose the next point to be 10 as points 2 to 9 are contained in a heavy line with 1 and thus have less overlap in the matching vector discussed later. For $U$, a $K_6$ chosen through 1, the viable 6-sets through 10 is reduced from the 344 candidates of $L_{10}$ as no pair of points contained in one of the six lines of $U$ are not permitted to be contained in another single line through 10 so some previously viable 6-sets must be discarded. \\
\\
Though we don't know much about the points from $B$ and their arrangement on $U$ we can discern information from which points of $B$ are in $U$. As 10 is contained in the heavy lines $B_3$ and $B_4$ each of the 6 single lines through 10 must contain one point of $B$ from each of $B_1$,$B_2$,$B_5$ and $B_6$. For any one of the 6 lines through 10 the two points of intersection with $B_5$ and $B_6$ will each be a point of intersection with a distinct line of $U$ as all the points of $B$ on those lines were used in the array above. They must be distinct lines or a line passing through 10 and a line passing through 1 would have two intersection points. As each of the 6 lines through 10 intersects 2 lines of $U$ in a point of $B$ they must intersect the remaining 4 lines of $U$ in $C$. Therefore in order to still be a potential single line a 6-set from $L_{10}$ should intersect four 6-sets of $U$ in $C$. \\

\begin{example}Here is one possible arrangement of the points in $B$ for the 6 lines through 10. It can be observed that each line intersects two of the lines through 1 shown above.
$$\begin{array}{ccccc}
1 & 76 & 82 & 100 & 107 \\
1 & 77 & 83 & 101 & 108 \\
1 & 78 & 84 & 102 & 109 \\
1 & 79 & 85 & 103 & 110 \\
1 & 80 & 86 & 104 & 111 \\
1 & 81 & 87 & 105 & 106
\end{array}$$
\end{example}

In order to ensure the $K_6$ from $L_{10}$ intersects $U$ correctly we introduce the concept of a matching vector. We let $m_i$ be the number of intersections a 6-set from $L_{10}$ has with the i-th row of $U$, then the matching vector is given by
$$\textbf{m}=m_1,m_2,\ldots,m_6.$$
For each $U$ the 6-sets that do not have a matching vector containing 4 ones and 2 zeros are discarded from $L_{10}$. All possible $K_6$s that can be generated from the remaining 6-sets of $L_{10}$ are found and we call each possible $K_6$ $V$ and record it as a pair with $U$. At this point the number of $\{U,V\}$ pairs varies depending on the set of representations chosen as options for $U$, but we had 16,205 in our calculation.  \\

Now that we have the lines through 1 and 10 we want to find the 6 lines that intersect $A$  at point 15 for maximum overlap with the two sets of lines we already have. Any single line through 15 intersects with two lines through 1 and two lines through 10 in points of $B$. This means for each choice of $U$ and $V$ we can use two matching vectors; one checking the intersection of 6-sets with $U$ and one checking intersections with $V$. If either matching vector does not contain 4 ones and 2 zeros then the 6-set is removed from $L_{15}$. As before all sets of 6 disjoint 6-sets are found from the refined $L_{15}$ and one is chosen as $W$. At this point we had 226 possibilities for $\{U,V,W\}$ and 96 different $U$s could be extended to find a $W$.

For each selection of $U$, $V$ and $W$ a $K_6$, $X$, from $L_{11}$ needs to be found. Here we have to modify our matching vector a little. Since the point 11 is on $C_2$ with 1 it still needs to have 4 ones and 2 zeros with respect to the intersections with $U$, but the 6-sets should intersect each 6-set of $V$ and $W$ 3 times. The lines through 11 intersect three of the lines in $V$ in the points contained in $B_1$, $B_2$ and $B_6$ and intersect three of the lines in $W$ in points contained in $B_1$, $B_2$ and $B_4$ leaving 3 intersections in $C$ for lines passing through 11 and $V$ and also their intersection with $W$. After eliminating unsuitable $L_{11}$ elements we create the possibilities for $K_6$ as usual and find 17 different collections of $U$, $V$, $W$ and $X$. None of the 17 can be extended to include $Y$, a $K_6$ from $L_{14}$ using the same method. Therefore it is not possible to construct the 90 single lines of the projective plane so our assumption that there is a vector of weight 15 in the projective plane of order 10 must be incorrect. 

\subsection{Computer Verification}
The results presented so far are those found by our computer search. The results given by MacWilliams, Sloane and Thompson in \cite{B15C} are slightly different beyond the 1021 representative $K_6$s as expected. This is because the existence of an entire plane containing any of the $K_6$s from an orbit is equivalent to the entire orbit being viable, but since we are looking at $K_6$s through point 10 first for our representation it might be a different set of $K_6$s under $G_1$. This makes confirming their search more difficult, but our search also came to the conclusion that no $K_6$, $Y$, could be found to extend the set of $U$, $V$, $W$ and $X$. We found 17 $\{U,V,W,X\}$s which is comparable with the result of 25 found by MacWilliams, Sloane and Thompson. The workbook for the search can be found at https://goo.gl/DUrCFJ 
\\

\renewcommand{\abstractname}{Acknowledgements}
\begin{abstract}
 I would like to thank Dr. Dillon Mayhew for his supervision and proofreading of this report and Dr. Ken Pledger for his useful feedback and proofreading.
\end{abstract}

\end{document}